\documentclass[10pt]{amsart}
\usepackage{amsmath, amsthm, amssymb,amscd}
\usepackage{graphicx}
\usepackage[v2]{xy}

\xyoption{all}
\newtheorem{corollary}{Corollary}

\newtheorem{lemma}{Lemma}
\newtheorem{theorem}{Theorem}

\newtheorem{remark}{Remark}

\long
\def\symbolfootnote[#1]#2{\begingroup
\def\thefootnote{\fnsymbol{footnote}}
\footnote[#1]{#2}\endgroup}

\begin{document}
\title{{A remark on the Isomorphism Conjectures}}
\date{}
\author{Crichton Ogle}
\address{ Department of Mathematics\\
The Ohio State University\\
Columbus, OH 43210\\
}
\email{ogle@math.ohio-state.edu}
\author{Shengkui Ye}
\address{Department of Mathematical Sciences\\
Xi'an Jiaotong-Liverpool University\\
Suzhou Industrial Park, Jiangsu Province\\
China 215123\\
}
\email{Shengkui.Ye@xjtlu.edu.cn}

\begin{abstract}
We show that for various natural classes of groups and appropriately defined 
$K$- and $L$-theoretic functors, injectivity or bijectivity of the assembly
map follows from the Isomorphism Conjecture being true for acyclic groups
lying within that class.
\end{abstract}

\maketitle

\symbolfootnote[0]{2010 \textit{Mathematics Subject Classification}. Primary
19Kxx; Secondary 57R67.} \symbolfootnote[0]{\textit{Key words and phrases}.
Baum-Connes Conjecture, Isomorphism Conjecture, acyclic groups.} \vskip.25in

\subsection*{Introduction}

A group $G$ is acyclic if the reduced homology $\tilde{H}_{\ast }(G;\mathbb{Z%
})=0$. It is well-known that every (torsion-free) group embeds as a subgroup
into a (torsion-free) acyclic group. It follows that Kaplansky's idempotent
conjecture (cf. \cite[p.~55]{mv}) holds for every torsion-free group if
and only it holds for every torsion-free acyclic group. Berrick, Chatterji
and Mislin \cite{bcm} prove that every (torsion-free) group $G$ embeds as a
subgroup into a (torsion-free) acyclic group $A(G)$ such that the conjugacy
relations are preserved, i.e. $g_{1}\underset{G}{\sim }g_{2}$ in $G$ if and
only if $g_{1}\underset{A(G)}{\sim }g_{2}$ in $A(G)$ for any two elements $%
g_{1},g_{2}\in G.$ This implies that the Bass conjecture (\cite{mv}, page
66) holds for any torsion-free group if and only if it holds for any
torsion-free acyclic group. In the note, we consider the isomorphism
conjectures, such as Baum-Connes conjecture and Farrell-Jones conjecture.
For more information on these conjectures, see Mislin-Vallete \cite{mv} and L%
\"{u}ck-Reich \cite{lr}. We prove that the fact that isomorphism conjectures
hold for any torsion-free acyclic group implies that the assembly maps are
injective for any torsion-free group. One interesting corollary is that the
isomorphism conjectures hold for any torsion-free group if and only the
assembly maps are surjective for any torsion-free group. \vskip.1in Note
that the isomorphism conjectures considered in this note are not the fibered
versions with coefficients (cf. \cite{blr}), which are stable under passage
to subgroups. Since every group embeds into an acyclic group, the
corresponding results for fibered isomorphism conjectures are obviously
true. \vskip.25in

\subsection*{Statement of results}

We will use the setup of \cite{dl}, with which we assume familiarity. For a
discrete group $G$, a set ${\mathcal{E}}$ of subgroups of $G$ is called a 
\emph{family of subgroups} if it is closed under conjugation and taking
subgroups. In other words, for any $H\in {\mathcal{E}},K\leq H$ and any $%
g\in G,$ we have $gHg^{-1}\in {\mathcal{C}}$ and $K\in {\mathcal{E}}$.
Typical examples of ${\mathcal{E}}$ are $\{1\}=$\{trivial subgroup\}; $%
\mathcal{F}in=$ \{finite subgroups\}; $\mathcal{VCY}=$ \{virtually cyclic
subgroups\}; $\mathcal{ALL}=$ \{all subgroups\}. For a family ${\mathcal{E}}$
of subgroups, the classifying space $E_{{\mathcal{E}}}(G)$ is uniquely
characterized up to equivariant homotopy by the property that the
fixed-point set $E_{{\mathcal{E}}}(G)^{H}$ is contractible for any $H\in {%
\mathcal{E}}$ and is empty for any $H\notin {\mathcal{E}}$. Let $H_{\ast
}^{G}(-;\mathbb{K}^{t})$ denote the equivariant homology theory associated
to the topological $K$-theory $\mathrm{Or}(G)$-spectrum $\mathbb{K}^{t}$.
Let $E_{\mathcal{F}in}(G)$ be the space classifying proper actions of $G$.
The Baum-Connes conjecture (as reformulated in \cite{dl}) asserts that the
assembly map 
\begin{equation}
H_{\ast }^{G}(E_{\mathcal{F}in}(G);\mathbb{K}^{t})\rightarrow K_{\ast
}^{t}(C_{r}^{\ast }(G))  \label{eqn:first}
\end{equation}%
is an isomorphism for all $\ast $, where the groups on the right are the
topological $K$-groups of the reduced $C^{\ast }$-algebra of $G$. We will
write BC for the Baum-Connes Conjecture, MBC resp.\ EBC for the statement
the assembly map in (\ref{eqn:first}) is a monomorphism resp.\ epimorphism ,
and $R$-BC (resp.\ $R$-MBC resp.\ $R$-EBC) for the conjecture that the
Baum-Connes assembly map becomes an isomorphism (resp.\ monomorphism resp.\
epimorphism) after tensoring both sides of (\ref{eqn:first}) with a subring $%
R\subseteq \mathbb{Q}$. Finally $R$-BC($G$) resp.\ $R$-MBC($G$) resp.\ $R$
-EBC($G$) will denote the conjecture that $R$-BC resp.\ $R$-MBC resp.\ $R$
-EBC holds for a particular group $G$. Let $\mathcal{G}$ be the class of all
groups. Given a subclass ${\mathcal{C}}\subset {\mathcal{ALL}}$, we say that 
$R$-IC, $R$-EC, or $R$-MC holds over $\mathcal{C}$ if the conjecture is true
for all groups in $\mathcal{C}$. The subclasses of interest here are: i) ${%
\mathcal{TF}}\subset {\mathcal{G}}$ consisting of all torsion-free discrete
groups and ii) $\mathcal{FF}\subset \mathcal{TF}$ the subcollection of
groups $G$ for which $BG\simeq X$ a finite complex (called $FF$ groups).

\begin{theorem}
Let $R$ be a subring of $\mathbb{Q}$. Let ${\mathcal{C}}={\mathcal{G}},{%
\mathcal{TF}}$ or ${\mathcal{FF}}$, the class of all groups, torsion-free
groups or groups with finite classifying spaces. If $R$-BC($G$) holds true
for all acyclic groups in $\mathcal{C}$, then $R$-MBC is true for all groups
in $\mathcal{C}$.
\end{theorem}

The assembly map considered above is a special case of a much more general
construction. For suitably defined functors $F$ on the class $\mathcal{G}$
of discrete groups, one has an assembly map 
\begin{equation}
HF_{\ast }(G)\rightarrow F_{\ast }(G)  \label{eqn:second}
\end{equation}%
and the Isomorphism Conjecture (IC) \cite{dl} asserts that this map is an
isomorphism, where $HF_{\ast }(_{-})$ denotes the appropriate homology group
associated to $F$. For any $G$, there is a unique $G$-map from $E_{\mathcal{F%
}in}(G)$ to a point. If $\mathcal{H}_{\ast }$ is any equivariant homology
theory (cf. \cite{lr}), then the assembly conjecture for the triple $%
\mathcal{H}_{\ast },\mathcal{F}in$ and $G$ asserts that the induced map from 
$\mathcal{H}_{\ast }(E_{\mathcal{F}in}(G))$ to $\mathcal{H}_{\ast }(\mathrm{%
pt})$ is an isomorphism, where $\mathrm{pt}$ denotes a point with trivial $G$%
-action. Following the definitions given above, the \textit{\ Epimorphism
Conjecture} (EC) resp.\ \textit{Monomorphism Conjecture} (MC) for the theory
being considered states that the assembly map in (\ref{eqn:second}) is a
monomorphism resp.\ epimorphism. Again, given a subring $R\subset \mathbb{Q}$%
, the conjecture $R$-IC resp.\ $R$-EC resp. $R$-MC is the conjecture that
the assembly map is an isomorphism resp.\ epimorphism resp.\ monomorphism
after tensoring with $R$, with the appendage \textquotedblleft ($G$%
)\textquotedblright\ indicating the conjecture for a particular group $G$.

\begin{theorem}
Let $F_{\ast }(G)=L_{\ast }^{<-\infty >}(\mathbb{Z}[G])$, with $HF_{\ast
}(G):=H_{\ast }^{G}(E_{{\mathcal{F}}in}(G);\mathbb{L}^{<-\infty >}(\mathbb{Z}%
))$ the equivariant homology group associated to the algebraic $L$-theory $%
\mathrm{Or}(G)$-spectrum $\mathbb{L}^{<-\infty >}(\mathbb{Z})$. Let ${%
\mathcal{C}}={\mathcal{G}},{\mathcal{TF}}$ or ${\mathcal{FF}}$. Fix $%
R\subset \mathbb{Q}$. If $\frac{1}{2}\in R$ and $R$-IC($G$) is true for the
functor $F$ for all acyclic groups in $\mathcal{C}$, then $R$-IC holds for $F
$ over $\mathcal{C}$. If $\mathcal{C}\subseteq {\mathcal{T}F}$, the
implication holds without restriction on $R$. In particular, the Novikov
Conjecture holds for all groups in $\mathcal{C}$ if the assembly map for $F$
is a rational isomorphism for all acyclic $G\in {\mathcal{C}}$.
\end{theorem}

Let $KH(S)$ denote the homotopy $K$-theory spectrum of the discrete ring $S$%
, as defined by Weibel in \cite{cw2}.

\begin{theorem}
Let $F_{\ast }(G)=KH_{\ast }(\mathbb{Z}[G])$, with $HF_{\ast }(G):=H_{\ast
}^{G}(E_{{\mathcal{F}}in}(G);\mathbb{KH}(\mathbb{Z}))$. Let ${\mathcal{C}}={%
\mathcal{G}},{\mathcal{T}F}$ or ${\mathcal{FF}}$. Let $R$ be a subring of $%
\mathbb{Q}$. If $R$-IC holds for $F$ for all acyclic groups in $\mathcal{C}$%
, then $R$-IC holds for $F$ over $\mathcal{C}$.
\end{theorem}

For ordinary algebraic $K$-theory, a slightly weaker result can be shown.

\begin{theorem}
For a discrete ring $S$, set $FS_{\ast }(G)=K_{\ast }(S[G])$, with $%
HFS_{\ast }(G):=H_{\ast }^{G}(E_{{\mathcal{F}}in}(G);\mathbb{K}(S))$. Let ${%
\mathcal{C}}={\mathcal{G}}$ or ${\mathcal{TF}}$ and $R$ an arbitrary subring
of $\mathbb{Q}$.

\begin{enumerate}
\item If $\mathbb{Q}$-IC holds for $F\mathbb{Z}$ for all acyclic groups in $%
\mathcal{C}$, then $\mathbb{Q}$-MC holds for $F\mathbb{Z}$ over $\mathcal{C}$%
.

\item Let $S$ be a regular ring containing the rationals $\mathbb{Q}$. If $R$%
-IC holds for $FS$ for all acyclic groups in $\mathcal{C}$, then $R$-MC
holds for $FS$ over $\mathcal{C}$.

\item Let $S$ be a regular ring. If $R$-IC holds for $FS$ for all acyclic
groups in ${\mathcal{FF}}$, then $R$-MC holds for $FS$ over ${\mathcal{F}L}$.
\end{enumerate}
\end{theorem}

\subsection*{Proof of Theorems 1, 2, and 3}

The proof in all cases is based on the method of \cite[\S 6.5]{jor}. For any
discrete group $G$, a classical construction allows us to embed $G$ in an
acyclic group $A(G)$ (its acyclic envelope), with the inclusion $%
i_{G}:G\hookrightarrow A(G)$ being functorial in $G$. Now the variation of
the Kan-Thurston construction detailed in \cite[Thm.~2.4]{jb} produces
a group $T(G)$ together with a surjective homomorphism $p_{G}:T(G)
\rightarrow G$ inducing an homology equivalence. The association $G\mapsto
T(G)$ is functorial in $G$; moreover $T(G)$ lies in the %
\underbar{Waldhausen-Cappell class} $\mathfrak{C}$ consisting of those
groups which can be constructed from free groups by i) amalgamated free
products, ii) HNN extensions, and iii) taking direct unions. Additionally,
as shown in \cite[Thm.~2.4]{jb}, starting with a group $G^{\prime }\in 
\mathfrak{C}$, the acyclic envelope $A(G^{\prime })$ can be formed so as to
remain inside of $\mathfrak{C}$. In the case $\mathcal{C}=\mathcal{G}$ or $%
\mathcal{TF}$, $A(T(G))$ will denote Block's construction of this envelope.
Let $A_{1}=G\times A(T(G))$, $A_{2}=A(T(G))$. There are inclusions 
\begin{gather}
T(G)\hookrightarrow A_{1},\,\,g\mapsto (p_{G}(g),i_{T(G)}(g)), \\
T(G)\hookrightarrow A_{2},\,\,g\mapsto i_{T(G)}(g).
\end{gather}
Let $A_{3}=A_{1}\underset{T(G)}{\ast }A_{2}$. By an application of
Mayer-Vietoris sequence, the group $A_{3}$ is acyclic.

In what follows, we will, for all of the functors considered above, write $%
HF_*(G)$ for $H_*^G(E_{{\mathcal{F}}in}G;\mathbb{F})$, where $\mathbb{F}$
denotes the $\mathrm{Or}(G)$-spectrum associated to $F$. There is a
homomorphism of sequences where the vertical arrows are given by assembly: 
\[
\xymatrixcolsep{.2in}
\xymatrix{
\dots \ar[r] & HF_{n+1}(A_3)\ar[r]^{\partial}\ar[d]^{\phi^3_{n+1}}
& HF_n(T(G))\ar[r]\ar[d]^{\phi^T_n} & HF_n(A_1)\oplus HF_n(A_2)\ar[r]\ar[d]^{\phi^1_n\oplus \phi^2_n} & HF_n(A_3)\ar[r]^{\partial}\ar[d]^{\phi^3_n}
& HF_{n-1}(T(G))\ar[r]\ar[d]^{\phi^T_{n-1}} &\dots\\
\dots \ar[r] & F_{n+1}(A_3)\ar[r]^{\partial}
& F_n(T(G))\ar[r] & F_n(A_1)\oplus F_n(A_2)\ar[r] & F_n(A_3)\ar[r]^{\partial} & F_{n-1}(T(G))\ar[r] &\dots
}
\]
\vspace{.1in}

As noted in \cite[p.~25]{mv}, the space $E_{\mathcal{F}in}(A_{3})$ is
equivalent (up to equivariant homotopy) to the homotopy push-out of the
diagram \vskip.2in 
\centerline{
\xymatrix{
A_{3}\underset{T(G)}{\times} E_{\mathcal{F}in}(T(G))\rto\dto &  A_{3}\underset{A_1}{\times}
E_{\mathcal{F}in}(A_{1}) \\ 
A_{3}\underset{A_2}{\times} E_{\mathcal{F}in}(A_{2}) &
}
} \vskip.2in by which one may derive the exactness of the top sequence for
coefficients in any $\mathrm{Or}(A_{3})$-spectrum. The commutativity of the
diagram, as well as the exactness of the bottom row, is the point that needs
to be verified. We consider first the case $\mathcal{C}=\mathcal{G}$ or $%
\mathcal{TF}$ for the functor $F_{\ast }(G)=K_{\ast }^{t}(C_{r}^{\ast }(G))$%
; here exactness of the bottom row follows by the results of Pimsner \cite%
{mp}, while the commutativity of the diagram has been shown by Oyono-Oyono 
\cite{oo}. As noted in \cite{jb}, the result of \cite{mp} implies $\phi
_{\ast }^{T}$ is an isomorphism. By the same reasoning, $\phi _{\ast }^{2}$
is an isomorphism, and $\phi _{\ast }^{3}$ is an isomorphism by hypothesis.
The five-lemma then implies $\phi _{\ast }^{1}$ must be an isomorphism as
well.

For a $\mathbb{Z}[\mathrm{Or}(G)]$-module $M$ and $G$-CW complex $X$, denote
by $H_{\ast }^{\mathrm{Or}(G)}(X;M)$ the Bredon homology of $X$ with
coefficients $M.$ Since the groups in $\mathfrak{C}$ are torsion-free, every
finite subgroup of $A_{1}$ is contained in $G$ and thus the family of finite
subgroups of $A_{1}$ is the same as that of $G$. Taking $M=\pi _{i}(\mathbb{K%
}^{top})$ viewed both as an $\mathbb{Z}[\mathrm{Or}(A_{1})]$-module and as
an $\mathbb{Z}[\mathrm{Or}(G)]$-module, one has isomorphisms 
\begin{eqnarray*}
H_{n}^{\mathrm{Or}(A_{1})}(E_{\mathcal{F}in}(A_{1});M) &\cong &H_{n}^{%
\mathrm{Or}(A_{1})}(E_{\mathcal{F}in}(G)\times E(A(T(G)));M) \\
&\cong &H_{n}^{\mathrm{Or}(G)}(E_{\mathcal{F}in}(G)\times \mathrm{B}%
A(T(G));M) \\
&\cong &H_{n}^{\mathrm{Or}(G)}(E_{\mathcal{F}in}(G);M).
\end{eqnarray*}%
By the equivariant Atiyah-Hirzebruch spectral sequence (cf. \cite{dl}),
there is an isomorphism 
\begin{equation*}
H_{n}^{A_{1}}(E_{\mathcal{F}in}(A_{1});\mathbb{K}^{top})\cong H_{n}^{G}(E_{%
\mathcal{F}in}(G);\mathbb{K}^{top}),\quad n\in \mathbb{Z}.
\end{equation*}%
Therefore, the inclusion map $G\rightarrow A_{1}$ induces an injection 
\begin{equation*}
\mathrm{Ker}(H_{n}^{G}(E_{\mathcal{F}in}(G);\mathbb{K}^{top})\rightarrow
K_{n}(C_{r}^{\ast }(G)))\subset \mathrm{Ker}(H_{n}^{A_{1}}(E_{\mathcal{F}%
in}(A_{1});\mathbb{K}^{top})\rightarrow K_{n}(C_{r}^{\ast }(A_{1}))).
\end{equation*}%
This implies that the assembly map $H_{n}^{G}(E_{\mathcal{F}in}(G);\mathbb{K}%
^{top})\rightarrow K_{n}(C_{r}^{\ast }(G))$ is injective, which completes
the proof of Theorem 1 for $R=\mathbb{Z}$. Tensoring with any ring flat over 
$\mathbb{Z}$ yields the same result for all $R\subset \mathbb{Q}$. \vskip.1in

For $\mathcal{C }= \mathcal{G}$ or $\mathcal{TF}$, the proofs of Theorems 2
and 3 follow exactly the same line of reasoning, after applying the
following modifications:

\begin{itemize}
\item In the case $F_*(G) = L^{<-\infty>}_*(\mathbb{Z}[G])$, the exactness
of the bottom row follows by the results of \cite{sc1}, the one complication
being the possible existence of $UNil$-terms. These terms vanish when
tensoring with any $R$ containing $\frac12$, or in the case the groups in
question are torsion-free. For this functor, the assembly map is an integral
isomorphism for groups in the class $\mathfrak{C}$ by \cite{sc1}, \cite{sc2}.

\item For $F_*(G) = KH_*(\mathbb{Z}[G])$, the corresponding results
(exactness of bottom row and equivalence of assembly map for $\mathfrak{C}$
-groups) has been shown in \cite{bl1}.

\item In both cases we have functoriality with respect to arbitrary group
homomorphisms, not just injective ones. The injection $G\rightarrowtail A_1$
of the first factor, the projection $A_1\twoheadrightarrow G$ onto the first
factor, and the naturality of the assembly map together allow us to conclude
that $R$-IC for the group $A_1$ implies $R$-IC for $G$.\hskip-.01in 
\footnote{%
In the case of the reduced $C^*$-algebra, it is unknown in general whether
the projection $A_1\twoheadrightarrow G$ defines an appropriate element of $%
KK(C^*_r(A_1),C^*_r(G))$. If it does, then the stronger conclusions of
Theorems 2 and 3 would apply as well to Theorem 1.}
\end{itemize}

\vskip.25in

We next consider the smaller class $\mathcal{FF}$. In order to duplicate the
above argument, the construction of the acyclic envelope requires
modification, as Block's construction does not preserve this class. Instead
(as in \cite{jor}), we use Leary's metric refinement of the Kan-Thurston
construction \cite{il}. To any complex $X$ Leary associates a locally CAT(0)
cubical complex $C(X)$ together with a map $p_{X}:C(X)\rightarrow X$ which
is an epimorphism on $\pi _{1}$ and an isomorphism in homology. The
association $X\mapsto (C(X),p_{X})$ is functorial in $X$; moreover if $X$ is
finite, so is $C(X)$. \vskip.1in

Let $G\in \mathcal{FF}$, and fix a finite basepointed complex $X_{G}$ with $%
X_{G}\simeq BG$. Let $\widehat{X_{G}}$ denote the cone on $X_{G}$; then the
canonical inclusion $X_{G}\hookrightarrow \widehat{X_{G}}$ is covered by an
inclusion of locally CAT(0) cubical complexes $C(X_{G})\hookrightarrow C(%
\widehat{X_{G}})$. Define the groups $A_{i},1\leq i\leq 3$ by 
\begin{gather*}
A_{1}:=G\times \pi _{1}(C(\widehat{X_{G}})); \\
A_{2}:=\pi _{1}(C(\widehat{X_{G}})); \\
A_{3}:=A_{1}\underset{\pi _{1}(C(X_{G}))}{\ast }A_{2},
\end{gather*}%
where $\pi _{1}(C(X_{G}))\hookrightarrow \pi _{1}(C(\widehat{X_{G}}))$ is
the inclusion of CAT(0)-groups\footnote{%
Leary shows that for any inclusion of complexes $X\hookrightarrow Y$, the
resulting inclusion $C(X)\hookrightarrow C(Y)$ is isometric and that the
image is a totally geodesic subcomplex of $C(Y)$, implying injectivity on $%
\pi _{1}$.} corresponding to the inclusion $X_{G}\hookrightarrow \widehat{%
X_{G}}$ and the inclusion $\pi _{1}(C(X_{G}))\hookrightarrow A_{1}$ is
similar to the inclusion $T(G)\hookrightarrow A_{1}$ defined in the first
paragraph of this proof. Writing $L_{\ast }^{<-\infty >}(\mathbb{Z}[H])$ as $%
L_{\ast }(\mathbb{Z}[H])$ and $H_{\ast }(BH;\mathbb{L}(\mathbb{Z}))$ simply
as $HL_{\ast }(BH)$, one has as before a commuting diagram of long-exact
sequences with the vertical maps induced by assembly: 
\[
\resizebox{6.5in}{.125in}{
\xymatrix{
\dots HL_{n+1}(BA_3)\ar[r]^{\partial}\ar[d]^{\psi^3_{n+1}}
& HL_n(B\pi_1(C(X_G)))\ar[r]\ar[d]^{\psi^C_n} & HL_n(BA_1)\oplus HL_n(BA_2)\ar[r]\ar[d]^{\psi^1_n\oplus \psi^2_n} & HL_n(BA_3)\ar[r]^(0.35){\partial}\ar[d]^{\psi^3_n}
& HL_{n-1}(B\pi_1(C(X_G)))\ar[d]^{\psi^C_{n-1}}\dots\\
\dots L_{n+1}(\mathbb Z[A_3])\ar[r]^{\partial}
& L_n(\mathbb Z[\pi_1(C(X_G))])\ar[r] & L_n(\mathbb Z[A_1])\oplus L_n(\mathbb Z[A_2])\ar[r] & L_n(\mathbb Z[A_3])\ar[r]^(0.35){\partial} & L_{n-1}(\mathbb Z[\pi_1(C(X_G))])\dots
}}
\]
\vspace{.1in}

Both $A_{2}$ and $\pi _{1}(C(X_{G}))$ are fundamental groups of finite
locally CAT(0) cubical complexes; it follows from the results of \cite{bl2}
that the assembly maps $\psi _{\ast }^{C}$ and $\psi _{\ast }^{2}$ are
isomorphisms. Moreover, $HL_{\ast }(BA_{1})\cong HL_{\ast }(BG)$, and so as
before one has an identification of kernels 
\begin{equation*}
\ker (\psi _{\ast }^{1})\cong \ker (HL_{\ast }(BG)\rightarrow L_{\ast }(%
\mathbb{Z}[G])).
\end{equation*}%
which, together with the injectivity of $\psi _{\ast }^{3}$ yields an
injection 
\begin{equation*}
\ker (HL_{\ast }(BG)\rightarrow L_{\ast }(\mathbb{Z}[G]))\cong \ker (\psi
_{\ast }^{1})\hookrightarrow {coker}(\psi _{\ast +1}^{3}).
\end{equation*}%
As all of the groups in the above diagram are objects in the category $%
\mathcal{FF}$, we arrive at the same conclusion as before. This completes
the proof of Theorem 2. In the case of the reduced group $C^{\ast }$-
algebra, the same argument for torsion-free groups applies, given that
groups acting properly on cubical CAT(0)-complexes satisfy the Haagarup
property \cite{NR}, and thus satisfy the Strong BC Conjecture by the work of
Higson-Kasparov \cite{hk}, which completes the proof of Theorem 1. \vskip.2in

Next we consider the statement of the third theorem when $\mathcal{C}=%
\mathcal{FF}$. For brevity, we say that $G$ satisfies condition $\mathcal{F}%
CAT$ if it acts properly, isometrically and cocompactly on a finite
dimensional CAT(0)-space\footnote{%
More precisely, one only needs $G$ to be in the class $\mathcal{B}$ as given
in \cite[Def.~1]{bl2} for this Lemma to apply.}.

\begin{lemma}
Suppose $G$ satisfies $\mathcal{F}CAT$. Then the natural transformation of
spectrum-valued functors $\mathbb{K}(_-)\to \mathbb{KH}(_-)$ from algebraic
to homotopy $K$-theory induces a weak equivalence 
\begin{equation*}
\mathbb{K}(\mathbb{Z}[G])\overset{\simeq}{\longrightarrow} \mathbb{KH}(%
\mathbb{Z}[G]).
\end{equation*}
\end{lemma}

\begin{proof}
For an arbitrary ring $A$, there exists a right half-plane spectral sequence
(cp.~\cite[Thm.~1.3]{cw2}]): 
\begin{equation*}
E_{pq}^{1}:=N^{p}K_{q}(A)\Rightarrow KH_{p+q}(A),\quad p\geq 0,q\in \mathbb{Z%
}.
\end{equation*}%
For $A=\mathbb{Z}[G]$ and $p>0$, the groups $N^{p}K_{\ast }(\mathbb{Z}[G])$
are summands of $K_{\ast }(\mathbb{Z}[G\times \mathbb{Z}^{p}])$. But if $G$
satisfies $\mathcal{F}CAT$, so does $G\times \mathbb{Z}^{p}$ for all $p\geq
0 $. Again, by the main result of \cite{bl2} and \cite{cw1}, these summands
identify isomorphically with the corresponding summands in the domain of the
Farrell-Jones assembly map, where they vanish. Thus for such groups, $%
N^{p}K_{\ast }(\mathbb{Z}[G])=0$ for all $p>0$, yielding the required
isomorphism on homotopy groups in all degrees.
\end{proof}

Thus the Farrell-Jones assembly map for $KH(_{-})$ - which for torsion-free
groups agrees with the classical assembly map $H_{\ast }(BG;\mathbb{K}(%
\mathbb{Z})\rightarrow KH_{\ast }(\mathbb{Z}[G])$ - is an isomorphism for $G$
satisfying $\mathcal{F}CAT$ (cf. \cite{cw1}). With this additional fact in
hand, the proof of Theorem 3 is complete. \vskip.2in

Unlike the reduced $C^{\ast }$ algebra, the full (or maximal) group $C^{\ast
}$ algebra is functorial with respect to arbitrary group homomorphisms. The
methods of the previous two theorems imply the following.

\begin{corollary}
There exist acyclic groups $G$ for which the assembly map 
\begin{equation*}
H_{\ast }^{G}(E_{\mathcal{F}in}(G);\mathbb{K}^{t})\rightarrow K_{\ast
}^{t}(C^{\ast }(G))
\end{equation*}%
fails to be an isomorphism, even rationally.
\end{corollary}

\begin{proof}
Suppose that for all acyclic groups the assembly maps are isomorphims. A
similar proof as those of Theorem 3 and 4 gives that the assembly map 
\begin{equation*}
H_{\ast }^{G}(E_{\mathcal{F}in}(G);\mathbb{K}^{t})\rightarrow K_{\ast
}^{t}(C^{\ast }(G))
\end{equation*}%
is an isomorphism for any $G.$ It is proved by Lafforgue \cite{La} that for some
infinite group $K$ with Kazhdan's property (T), the Baum-Connes assembly map 
\begin{equation*}
H_{\ast }^{G}(E_{\mathcal{F}in}(K);\mathbb{K}^{t})\rightarrow K_{\ast
}^{t}(C_{r}^{\ast }(K))
\end{equation*}%
is an isomorphism. Since the latter map factors through the former (cp.~\cite[p.~83]{mv}), we have an isomorphism%
\begin{equation*}
K_{\ast }^{t}(C^{\ast }(K))\cong K_{\ast }^{t}(C_{r}^{\ast }(K)).
\end{equation*}%
However, it is well-known that for any infinite group with property (T) these
groups are not isomorphic, even rationally (cp. \cite[Cor.~3.1]{JV} and its
proof). This gives an contradiction.
\end{proof}

\vskip.2in

Finally we consider the statement of Theorem 4. Here the results of
Waldhausen \cite{fw} produce a Mayer-Vietoris type of long-exact sequence
which appears as the bottom row in the commuting diagram
\vskip.1in
\centerline{
\resizebox{6.5in}{.125in}{
\xymatrix{
\dots HFS_{n+1}(A_3)\ar[r]^{\partial}\ar[d]^{\phi^3_{n+1}}
& HFS_n(T(G))\ar[r]\ar[d]^{\phi^T_n} & HFS_n(A_1)\oplus HFS_n(A_2)\ar[r]\ar[d]^{\phi^1_n\oplus \phi^2_n} & HFS_n(A_3)\ar[r]^(0.425){\partial}\ar[d]^{\phi^3_n}
& HFS_{n-1}(T(G))\ar[d]^{\phi^T_{n-1}} \dots\\
\dots K_{n+1}(S[A_3])\ar[r]^(0.35){\partial}
& K_n(S[T(G)])\oplus Nil_n(T(G),A_1,A_2)\ar[r] & K_n(S[A_1])\oplus K_n(S[A_2])\ar[r] & K_n(S[A_3])\ar[r]^(0.27){\partial} & K_{n-1}(S[T(G)])\oplus Nil_n(T(G),A_1,A_2) \dots
}}
}

 \vskip.1in

Assume that the Farrell-Jones assembly map is an isomorphism for any acyclic
group. Then $\phi _{\ast }^{2}$ and $\phi _{\ast }^{3}$ are isomorphisms.
When either $\mathbb{Q}\subset S$, or $\mathbb{K}$ represents rationalized
algebraic $K$-theory with $S=\mathbb{Z}$, the Farrell-Jones assembly map is
injective for any group in the Waldhausen-Cappell class $\mathfrak{C}$ \cite%
{bl1}. With $\phi _{\ast }^{T}$ injective, the map $\phi _{n}^{1}$ is
injective by a diagram chase. This shows that the kernel 
\begin{equation*}
\mathrm{Ker}(H_{n}^{G}(E_{\mathcal{F}in}(G);\mathbb{K})\rightarrow
K_{n}(S[G]))\subset \mathrm{Ker}(\phi _{n}^{1})
\end{equation*}
is trivial. When $G\in \mathcal{FF}$, we produce $A_{1},A_{2}$ and $A_{3}$
using locally CAT(0) cubical complexes as before. The group $\pi
_{1}(C(X_{G}))$ acts properly and cocompactly on the universal cover of $%
C(X_{G}),$ which is a CAT(0) cubical complex. According to \cite{cw1}, the
Farrell-Jones conjecture is true for $\pi _{1}(C(X_{G}))$ with any
coefficients. Using a similar diagram chasing, we see that $\mathrm{Ker}%
(H_{n}^{G}(E_{\mathcal{F}in}(G);\mathbb{K})\rightarrow K_{n}(S[G]))=0$ in
(3). The rational algebraic $K$-theory with $R=\mathbb{Z}$ is proved
similarly, completing the proof of Theorem 4. \vskip.2in

For a torsion-free acyclic group $A$, there are isomorphisms $H_{n}^{A}(E_{%
\mathcal{F}in}(A);\mathbb{F})\cong H_{n}(\mathrm{B}A;\mathbb{F}(A/e))\cong
H_{n}(\mathrm{pt};\mathbb{F}(A/e)),$ where $e$ denotes the trivial subgroup
of $A$. This implies that the assembly map is injective for a torsion-free
acyclic group. Therefore we have the following.

\begin{corollary}
Following Theorems 1 and 4,

\begin{enumerate}
\item[(1)] The Baum-Connes conjecture is true for every torsion-free group
if and only if the Baum-Connes assembly map is an epimorphism for every
torsion-free group.

\item[(2)] Let $S$ be a regular ring with $\mathbb{Q}\subset S.$ The
Farrell-Jones conjecture with coefficients in $S$ (resp. the rational
Farrell-Jones conjecture with coefficients in $\mathbb{Z}$) holds for every
torsion-free group if and only if the integral (resp. rational) assembly map
is an epimorphism for every torsion-free group.

\item[(3)] Let $S$ be a regular ring. The Farrell-Jones conjecture is true
for every $FF$ group (with coefficients in $S$) if and only if the assembly
map is an epimorphism for every $FF$ group (with coefficients in $S$).
\end{enumerate}
\end{corollary}

\begin{remark}{\rm
It is currently unknown whether the original Baum-Connes Conjecture holds
for CAT(0)-groups of the type considered by Bartels and L\"uck in \cite{bl2}%
. However, based on the results of \cite{gy} and \cite{bcgnw}, it seems
plausible that similar results as those above can be obtained for the Coarse
Baum-Connes Conjecture. We hope to address these issues more completely in
future work.}
\end{remark}

\vskip.1in

\bigskip

\noindent \textbf{Acknowledgements }We would like to thank Ian Leary for
helpful correspondence during the preparation of this paper. We also thank
the referee for detailed comments on a previous version. The second author
is supported by Jiangsu Natural Science Foundation No. BK20140402 and NSFC
No. 11501459.\newpage\

\end{document}